\documentclass[12pt,a4paper,reqno]{amsart}
\usepackage{latexsym}
\usepackage{amssymb}
\usepackage{enumitem}

\usepackage[active]{srcltx}

\def \zmontar{\buildrel}

\def \za{\alpha}
\def \zb{\beta}
\def \zg{\gamma}

\def \ze{\varepsilon}

\def \zh{\theta}

\def \zl{\lambda}
\def \zm{\mu}

\def \zp{\pi}

\def \zt{\tau}

\def \zf{\varphi}

\def \zS{\Sigma}

\def \zF{\Phi}

\def \zlma{\ell}

\def \zsu{\sum}

\def \zin{\cap}

\def \zun{\cup}
\def \zung{\bigcup}

\def \zpor{\times}

\def \zmei{\leq}
\def \zmai{\geq}
\def \zco{\subset}

\def \zpe{\in}

\def \zeq{\equiv}

\def \znoi{\neq}

\def \znope{\not\in}

\def \zpar{\partial}
\def \zinf{\infty}
\def \zva{\emptyset}

\def \zfl{\rightarrow}

\def \z/{\over}

\newcommand {\CC}{\mathbb C}

\newcommand {\RR}{\mathbb R}
\newcommand {\TT}{\mathbb T}
\newcommand {\ZZ}{\mathbb Z}
\newcommand {\NN}{\mathbb N}

\newcommand {\F}{\mathcal F}

\newcommand {\W}{\mathcal W}
\newcommand {\U}{\mathcal U}
\newcommand {\GG}{\mathbf G}

\newcommand {\co}{\colon}

\newtheorem{theorem}{Theorem}[section]
\newtheorem*{theorem*}{Theorem}
\newtheorem{lemma}[theorem]{Lemma}
\newtheorem{corollary}[theorem]{Corollary}
\newtheorem*{corollary*}{Corollary}
\newtheorem{proposition}[theorem]{Proposition}

\newtheorem*{example*}{Example}

\theoremstyle{definition}
\newtheorem{remark}[theorem]{Remark}
\newtheorem{example}[theorem]{Example}

\title[Minimal dimension of the orbits of a $\mathbb R^n$-action]
{On the minimal dimension of the orbits of a $\mathbb R^n$-action}

\author{Francisco-Javier~Turiel}
\address[F.J.~Turiel]
{Geometr{\'\i}a y Topolog{\'\i}a,
Facultad de Ciencias,
Campus de Teatinos, s/n,
29071-M{\'a}laga, Spain}
\email[F.J.~Turiel]{turiel@uma.es}

\begin{document}

\begin{abstract}
Consider a smooth action of $\mathbb R^n$ on a connected manifold $M$, not necessarily
compact, of dimension $m$ and rank $k$. Assume that $M$ is not a cylinder. Then there 
exists an orbit of the action of dimension $<(m+k)/2$.
As a consequence, one shows that if there is a non-zero element of the ring of
Pontrjagin classes of $M$ of degree $4\ell\geq 4$, then there exists an orbit
of the action of dimension $\leq m-\ell-1$.
\end{abstract} 
 \maketitle


{\em 2020 Mathematics Subjet Classification:} Primary 37C85; Secondary 57R25.

\section{Introduction} \label{sec-1}

In this work manifolds (always without boundary) and their associated objects are real 
$C^\zinf$ unless another thing is said.

The degeneracy of the orbits of an action of a Lie group on a manifold is nowadays a classical
question. Here we will consider the case of $\RR^n$ acting on open manifolds.
A result by E. Lima \cite{LI}, of 1964, states that any action of $\RR^n$ on a compact 
surface of non-vanishing Euler characteristic has a fixed point. Later on P. Molino and
the author \cite{MT8} have given, for compact manifolds, an upper bound for the minimal
dimension of the orbits of a $\RR^n$-action, which includes the Lima's theorem (in
$C^\zinf$ class) as a particular case. The goal of this paper is to extend this result to
open manifolds.

Other results of this kind for symplectic actions can be seen in \cite{TU}. For some quite 
interesting  results in dimension three see \cite{BS}.

Recall that the {\em span} of a manifold $P$ is the maximal number of vector fields on $P$
that are linearly independent at each point. The same notion but for commuting vector
fields is named the {\em rank} of $P$ (defined by J.W. Milnor at the Seattle Topology
Conference of 1963 and echoed by S.P. Novikov \cite{NO}). Finally, if the vector fields 
commute and are complete this number is called the {\em file} of $P$ (introduced by
H. Rosenberg \cite{RO}). Obviously for compact manifolds file equals rank.

A manifold diffeomorphic to $\TT^r \zpor\RR^{m-r}$ will be called a {\em cylinder of type
$r$ and dimension $m$}. This is our  main result. 

\begin{theorem}\label{T11}
Consider an action of $\RR^n$ on a connected manifold $M$ of dimension $m$ and
rank $k$. Assume that $M$ is not a cylinder. Then there is an orbit of the action 
of dimension $<(m+k)/2$.
\end{theorem}

(The compact case of the foregoing theorem is just the main result
of \cite{MT8}.)

As a consequence, we show that if there is a non-zero element of the ring of
Pontrjagin classes of $M$ of degree $4\ell\geq 4$, then there exists an orbit
of the action of dimension $\leq m-\ell-1$ (Proposition \ref{P57}).

This work consists of six sections, the first one being the introduction. Sections
\ref{sec-2} and \ref{sec-3} are devoted to state some results needed later on. 
The proof of Theorem \ref{T11}, the main result, is given in Section 4.
In Section 5 we illustrate Theorem \ref{T11} with several examples and
one proves Proposition \ref{P57}. Finally in Section 6 one studies the $\RR^n$-actions
whose orbits have codimension$\zmei 1$.

For general questions on Lie groups actions see  \cite{PA},  for those on characteristic 
classes see \cite{MS} and finally \cite{GO} for questions on foliations.

\section{Preliminary results} \label{sec-2}
Let $V$ be a real vector space of dimension $n$ and let $\GG(k)$ be the 
Grassmann manifold
of $k$-planes  of $V$ (if necessary one will write $\GG(V,k)$ instead of $\GG(k)$). 
Given a $k'$-plane $F$ of $V$ set
$$N_F(k,r)=\{E\zpe\GG(k)\co \dim(E\zin F)=r\},\, r\zmai 1,\,\,{\rm{and}}\,\,
N_F(k)=\zung_{r\zmai 1}N_F(k,r).$$ 

Every $N_F(k,r)$ is a regular (embedded) submanifold of $\GG(k)$
of codimension $r(n+r-k-k')$.

Consider a manifold $P$ and a differentiable map $f\co P\zfl\GG(k)$. One will say that $f$
is {\em transverse} to $N_F(k)$ if it is transverse to each stratum $N_F(k,r)$, $r\zmai 1$.

\begin{lemma}\label{L21} One has:

\begin{enumerate} [label={\rm (\roman{*})}]
\item The set of those $F\zpe\GG(k')$ such that  $f$ is not transverse to $N_F(k)$
is of the first category in $\GG (k')$.
\item Assume that $P'$ is a subset of $P$ of the first category. If 
$\dim P\zmei r_0(n+r_0 -k-k')$, where $n,k,k',r_0 \zmai 0$, then the set
$$D=\{F\zpe\GG(k')\co {\rm there\, exists}\,\, p\zpe P' \, {\rm such\, that}\,  
\dim(f(p)\zin F)\zmai r_0 \}$$
is of the first category in $\GG(k')$.
\end{enumerate}
\end{lemma}

Consider a foliation $\F$ of dimension $k$ on a $m$-manifold $M$. A transversal $T$ to 
$\F$ is called {\em neat} if no leaf of $\F$ intersects $T$ two or more times. A point
of $M$ is said {\em wandering} if it belongs to some neat transversal. The set $\W(\F)$ 
of all wandering points, that is the union of all neat transversals, is open and $\F$-saturated.

Now assume the existence of an action of $\RR^n$ on $M$ whose orbits are just the leaves
of $\F$. In this case the leaves of $\F$ are cylinders $C_r =\TT^r \zpor\RR^{k-r}$,
$r=0,\dots,k$. The number $r$ one will be called the {\em type of the leaf}. Assigning
to each point the type of the leaf passing through it defines a function
$\zt\co M\zfl\NN$, which is constant on the leaves of $\F$.

\begin{lemma}\label{L22}
The set of those points of $\W(\F)$ such that $\zt$ is constant around them is $\F$-saturated, 
open and dense in $\W(\F)$.
\end{lemma}

\begin{theorem}\label{T23}
Assume that on a neighborhood of a point $p\zpe\W(\F)$ the function $\zt$ is constant.
Set $r=\zt(p)$. Then there exists an open set of $M$ identified to $D\zpor C_r$, where
$D$ is an open disc of $\RR^{m-k}$ centered at the origin, such that:
\begin{enumerate}[label={\rm (\alph{*})}]
\item
$D\zpor C_r$ is $\F$-saturated, every $\{x\}\zpor C_r$ is a leaf of $\F$ and $\{0\}\zpor C_r$
is that passing through $p$. 
\item
One can choose coordinates $(x,\zh,y)$ on $D\zpor C_r =D\zpor\TT^r \zpor\RR^{k-r}$ in such
a way that the fundamental vector fields of the action write  
$$X=\zsu_{j=1}^r f_j (x){\frac{\zpar} {\zpar\zh_j}}   
+\zsu_{\zlma=1}^{k-r} g_\zlma (x){\frac{\zpar} {y_\zlma}}\, .$$
\end{enumerate}
\end{theorem}

For the proof of the foregoing results see \cite{TU} pages 271, 272, 287 and 288.

\section{Some  results on vector fields} \label{sec-3}
In this section one states some technical results needed later on.

\begin{lemma}\label{L31}
Let $N$ be a regular submanifold of a manifold $M$. Assume that $N$ is a closed subset 
of $M$. Consider a vector field $X$ on $N$. Then there exists a vector field $X'$ on $M$
that is complete on $M-N$ and equals $X$ on $N$.
\end{lemma}

\begin{proof}
Consider a tubular neighborhood $\zp\co E\zfl N$ of $N$ endowed, as vector bundle, with
a Riemannian metric.

Set $S_\ze =\{x\zpe E\co\| x\|<\ze\}$ and $T_\ze =\{x\zpe E\co\| x\|=\ze\}$, $\ze>0$.

Let $Y$ be a vector field on $E$ tangent to each $T_\ze$ and such that $\zp_* Y=X$
(its existence is obvious). Consider a family $\{A_k \}_{k\zmai 1}$ of open sets of $N$
such that every $\overline A_k$ is a compact set included in $A_{k+1}$ and
$\zung_{k\zmai 1} A_k =N$.

Set $B_k =\zp^{-1}(A_k )\zin S_{1/k}$; every $B_k$ is an open set of $E$, so of M, 
with compact adherence. Let $B=\zung_{k\zmai 1}B_k$.

Since manifolds are normal spaces, there is a closed set $C$ of $M$ such that
$C\zin N=\zva$ and $M-E\zco\zmontar{\circ}\over{C}$. Consider a function 
$\zf\co M\zfl[0,1]$ such that $\zf(N)=1$ and $\zf(C\zun (M-B))=0$. Now set
$X'=\zf Y$ on $E$ and $X'=0$ on $\zmontar{\circ}\over{C}$. As 
$E\,\zun\zmontar{\circ}\over{C} =M$ and $\zf$ vanishes on 
$E\,\zin\zmontar{\circ}\over{C}$ the vector field $X'$ is well defined; moreover $X'=X$ on $N$.

For finishing it is enough to show that $X'$ is complete on $E-N$ since it vanishes on
$\zmontar{\circ}\over{C}$. As $X'$ is tangent to each $T_\ze$, $\ze >0$, it suffices
to show that it is complete on $T_\ze$. But $X'|T_\ze$ has compact support because
$X'$ vanishes on $T_\ze -\zp^{-1}(A_k )$ for any $k$ such that $1/k<\ze$.
\end{proof}

\begin{corollary}\label{C32}
Consider a regular submanifold $N$ of a manifold $M$, where $N$ is a closed subset of $M$.
Let $X$ and $\zF_t$ be a
vector field on $N$ and its flow respectively. Assume that $\zF_1$ is defined on the whole $N$. 
Then $M$ and $(M-N)\zun\zF_1 (N)$ are diffeomorphic.
\end{corollary}

Indeed, let $X'$ be a vector field like in Lemma \ref{L31} and let $\zF'_t$ be its flow. Then
$\zF'_1 \co M\zfl (M-N)\zun\zF_1 (N)$ is a diffeomorphism.

\begin{lemma}\label{L33}
Given $a>0$ there exists a vector field $X$ on $\RR^m$ such 
that $\zF_1 (\RR^m )=B_a (0)$ where $\zF_t$ is the flow of $X$.
\end{lemma}

By taking into account the diffeomorphism 
$x\zpe\RR^m -\{0\}\zfl(\| x\| /a,x/\| x\|)\zpe\RR^+ \zpor S^{m-1}$,
for proving this result it suffices to exhibit a vector field 
$Z=g(u)\zpar/\zpar u$ on $\RR^+$ vanishing near zero such that 
$\tilde\zF_1 (\RR^+ )=(0,1)$, where  $\tilde\zF_t$ is the flow of $Z$.
For example consider a non-positive function $g\co\RR^+ \zfl\RR$ such that
$g((0,1/4))=0$ and $g(u)=-u^2$ on $(1/3,\zinf)$.
\medskip

In the next result, balls in $\RR^m$ centered at the origin and the radius $r$ are
denoted $B_r$ (open) or ${\overline B}_r$ (closed) and those in $\RR^n$ centered at 
the origin $B'_r$ and ${\overline B'}_r$ respectively. Similarly, $S_r$ and $S'_r$ are
the spheres of radius $r$ centered at the origin. On the other hand,
$(x,y)$ will be a point of $\RR^m\zpor\RR^n$.

\begin{lemma}\label{L34}
Given real numbers $0<c<a$ and $0<d<b$ set
$$A=B_a \zpor B'_b -\{0\}\zpor\left( B'_b -B'_d \right)
\quad {\rm and}\quad
A'=A-{\overline B}_c \zpor S'_d.$$

Then there exist $\ze>0$ and a diffeomorphism $\zf\co A\zfl A'$  such that $\zf(x,y)=(x,y)$
if $\| x\|\zmai a-\ze$ or $\| y\|\zmai b-\ze$ (or both).
\end{lemma}

\begin{proof}
First observe that there exist $0<\ze<\min\{(a-c)/2, (b-d)/2\}$ 
and a vector field $X$ on $B_a -\{0\}$, whose
flow $\zF_t$ is defined everywhere for $t=1$, such that:
\begin{itemize}
\item $\zF_1 (B_a -\{0\})=B_a -{\overline B}_c$,
\item $X$ vanishes on $B_a -B_{a-2\ze}$.
\end{itemize}

Indeed, identify $\RR^m -\{0\}$ to $\RR^+ \zpor S^{m-1}$ in the usual way
($x\zfl (\| x\|,x/\| x\|)$) and endow this last space with the product coordinates
$(u,v)$. Then set $X=g(u)\zpar/\zpar u$ where $g\co (0,a)\zfl\RR$ is a non-negative
function, $g((0,c])=c$ and $g([a-2\ze,a))=0$.

Now consider $X$ as a vector field on $(B_a -\{0\})\zpor S'_d$ tangent to the first factor.
Then its flow $\zF_t$ is defined everywhere for $t=1$ and:
\begin{itemize}
\item $\zF_1 ((B_a -\{0\})\zpor S'_d)=(B_a -{\overline B}_c )\zpor S'_d$,
\item $X$ vanishes on $(B_a -B_{a-2\ze})\zpor S'_d$.
\end{itemize}

Since $(B_a -\{0\})\zpor S'_d$ is both a closed subset and a 
regular submanifold of $A$, by Lemma \ref{L31}
the vector field $X$ extends to a vector field $X'$ on $A$, which is complete outside
$(B_a -\{0\})\zpor S'_d$.

Let $h\co A\zfl[0,1]$ be a function such that:
\begin{itemize}
\item $h(x,y)=1$ if $\| x\|\zmei a-2\ze$ and $\| y\|\zmei b-2\ze$,  
\item $h(x,y)=0$ if $\| x\|\zmai a-\ze$ or $\| y\|\zmai b-\ze$.
\end{itemize}
 
 Set $X''=hX'$ and let $\zF''_t$ be its flow. Then $X''=X$ on  $(B_a -\{0\})\zpor S'_d$.
Moreover $X''$ is complete outside  $(B_a -\{0\})\zpor S'_d$, it vanishes if
$\| x\|\zmai a-\ze$ or $\| y\|\zmai b-\ze$, and $\zF''_1 (A)=A'$. Therefore for finishing
the proof set $\zf=\zF''_1$.
\end{proof}

\section{Proof of Theorem \ref{T11}} \label{sec-4}
Consider an action on the left of $\RR^n$ on a connected manifold $M$ of dimension $m$
and rank $k<m$ (i.e. $M$ is not a cylinder). Let $V$ be the Lie algebra of $\RR^n$ and 
let $X_v$, $v\zpe V$, be the fundamental vector field on $M$ associated to $v$.

For each subset $V'$ of $V$ and each point $p\zpe M$ set
$V'(p)=\{X_v (p)\co v\zpe V'\}$.

As usual, the infinitesimal isotropy of a point $p$ is the set $I(p)=\{v\zpe V\co X_v (p)=0\}$.
Since $V$ is abelian, $I(p)$ only depends on the orbit of $p$.

Denote by $\zS_r$, $r=0,\dots,m$, the set of those points of $M$ whose orbit has dimension $r$; 
that is $p\zpe\zS_r$ if and only if $\dim I(p)=n-r$. Let $h_r \co\zS_r \zfl\GG(n-r)$ be the map
given by $h_r (p)=I(p)$.

A chart $(U,x_1 ,\dots,x_m )$ is said {\em $\zS_r$-adapted}  if the image of $U$ on
$\RR^m$ is a product of open intervals and there exist vectors 
$v_1 ,\dots,v_r \zpe V$ such that $X_{v_j}=\zpar/\zpar x_j$, $j=1,\dots,r$.

It is easily seen that every point of $\zS_r$ belongs to the domain of some $\zS_r$-adapted
chart. Therefore $\zS_r$ can be covered by a countable family of $\zS_r$-adapted charts.
(Here countable includes the finite case.)

\begin{lemma}\label{L41}
For every $r=0,\dots,m$ the map $h_r \co\zS_r \zfl\GG(n-r)$ is differentiable of rank
$\zmei m-r$, i.e. it can be locally extended to a differentiable map of rank $\zmei m-r$
at every point of its domain.
\end{lemma}

\begin{proof}
It suffices to show that $h_r \co U\zin\zS_r \zfl\GG(n-r)$ is differentiable and of rank $\zmei m-r$
for any $\zS_r$-adapted chart $(U,x_1 ,\dots,x_m )$. Consider vectors $v_1 ,\dots,v_r \zpe V$ 
such that $X_{v_j}=\zpar/\zpar x_j$, $j=1,\dots,r$ and choose $v_{r+1} ,\dots,v_n \zpe V$ in
a such a way that $\{v_1 ,\dots,v_n\}$ is a basis of  $V$. Then each
$X_{v_i}=\zsu_{\zlma=1}^m f_{i\zlma}\zpar/\zpar x_\zlma$, $i=r+1,\dots,n$.

As $V$ is an abelian Lie algebra functions $f_{i\zlma}$ only depend on $(x_{r+1},\dots,x_m)$.

Now consider the open set $A\zco\GG(n-r)$ of those $(n-r)$-planes that have as a basis
$$\left\{v_{r+1}-\zsu_{\zlma=1}^r a_{r+1 \zlma}v_\zlma ,\dots,
v_{n}-\zsu_{\zlma=1}^r a_{n \zlma}v_\zlma \right\},$$
where $a_{i\zlma}$, $i=r+1,\dots,n$, $\zlma=1,\dots,r$, are real numbers. Then
$$\left(A, (a_{i\zlma}), i=r+1,\dots,n, \zlma=1,\dots,r\right)$$
is a system of coordinates of $\GG(n-r)$. 

Let $\tilde h_r \co U \zfl A\zco\GG(n-r)$ be the differentiable map defined in coordinates 
by $\tilde h_r (x)=(f_{i\zlma}(x))$, $i=r+1,\dots,n$, $\zlma=1,\dots,r$. Then $\tilde h_r$, which
only depends on $(x_{r+1},\dots,x_m )$, extends $h_r \co U\zin\zS_r \zfl \GG(n-r)$
since if $x\zpe U\zin\zS_r$, then $V(x)$ is spanned by $\zpar/\zpar x_1 ,\dots,\zpar/\zpar x_r$
and 
$$\left\{v_{r+1}-\zsu_{\zlma=1}^r f_{r+1 \zlma}(x)v_\zlma ,\dots,
v_{n}-\zsu_{\zlma=1}^r f_{n \zlma}(x)v_\zlma \right\}$$
is a basis of $I(x)$.
\end{proof}

\begin{proposition}\label{P42}
There always exists some orbit of dimension $\zmei (m+k)/2$. 
\end{proposition}

\begin{proof}
Assume that all the orbits have dimension $>(m+k)/2$. Let $(U,x_1 ,\dots,x_m )$ be a 
$\zS_r$-adapted chart, $r>(m+k)/2$, and let $\tilde h_r \co U\zfl\GG(n-r)$ be like in the proof of
Lemma \ref{L41}. Consider $F\zpe\GG(k+1)$ such that $\tilde h_r$ is transverse to
$N_F (n-r)$. Then $Im(\tilde h_r )\zin N_F (n-r)=\zva$; indeed, otherwise as the rank of
$\tilde h_r \zmei m-r$ and the smallest codimension among the strata of $N_F (n-r)$ is
that of $N_F (n-r,1)$ one has
$$m-r\zmai {\rm codim}\, N_F (n-r,1)=r-k,$$
hence $r\zmei (m+k)/2$, {\em contradiction}. 

Observe that from $Im(\tilde h_r )\zin N_F (n-r)=\zva$ follows
$I(q)\zin F=\{0\}$, that is to say $\dim F(q)=k+1$, for every $q\zpe U\zin\zS_r$.

But $M=\zung_{r>(m+k)/2}\zS_r$ may be covered by a countable family of adapted charts, and
assertion (i) of Lemma \ref{L21} allows us to choose $F\zpe\GG(k+1)$ such that each map 
$\tilde h_r$  associated to this countable family is transverse to $N_F (n-r)$. Therefore
$\dim F(p)=k+1$ for all $p\zpe M$ and ${\rm rank}M\zmai k+1$, {\em contradiction}. 
\end{proof}

Proposition \ref{P42} is due to D. Simen \cite{SI}; see \cite{MT8} as well.

\begin{remark}\label{R43}
In fact one has proved two things slightly stronger:
\begin{itemize}
\item In the proposition above the rank of $M$ can be replaced by the file.
\item Proposition \ref{P42} still holds for local actions.
\end{itemize}
\end{remark}

{\em From now on and until the end of Section \ref{sec-4} one will assume that there is no
orbit of dimension $<(m+k)/2$.We will reach a contradiction.} 
More precisely, one will construct an open subset of $M$ of rank $\zmai k+1$ that is
diffeomorphic to $M$.

Since by Proposition \ref{P42} there exist orbits of
dimension $\zmei (m+k)/2$ necessarily $m+k$ is even. Let $s=(m+k)/2$.

\subsection{Existence of suitable elements in $\GG(k+1)$} \label{sec-41}

On the open set ${\zmontar{\circ}\over{\zS}_s} \zco M$ the action of $\RR^n$ defines a 
foliation $\F$ of dimension $s$ and codimension $(m-k)/2$. Let $\W$ be the set of its
wandering points and let $\U$ be  the set of those points of $\W$ around of which the type
of the leaf is constant. By Lemma \ref{L22}, $\U$ is $\F$-saturated, open and dense in $\W$,
so $\W -\U$ is of the first category in $\W$.

The aim of this subsection is to prove the following result:

\begin{proposition}\label{P44}
There  exists $F\zpe\GG(k+1)$ such that:

\begin{enumerate} [label={\rm (\arabic{*})}]
\item  $h_s \co{\zmontar{\circ}\over{\zS}_s}\zfl\GG(n-s)$ is transverse
to $N_F (n-s)$ and $h^{-1}_s (N_F (n-s,\zlma))=\zva$ for every $\zlma\zmai 2$.
Therefore  $h^{-1}_s (N_F (n-s))= h^{-1}_s (N_F (n-s,1))$.
\item $h^{-1}_s (N_F (n-s))= h^{-1}_s (N_F (n-s,1))$ is the singular set of $F$, i.e.
the set of those points $p\zpe M$ such that  $\dim F(p)<k+1$.

\noindent Moreover, if $\dim F(q)<k+1$ for some $q\zpe M$ then $\dim F(q)=k$ and $q\zpe \U$.
\item   $h^{-1}_s (N_F (n-s))= h^{-1}_s (N_F (n-s,1))$ is empty or, both a closed subset and a 
regular submanifold of $M$ of dimension $s$ which is included in $\U$. Moreover
$h^{-1}_s (N_F (n-s))$ is saturated for the action of $\RR^n$. 
\end{enumerate}
\end{proposition}

An element of $\GG(k+1)$ like in the proposition above will be called {\em suitable}.

Consider a $\zS_s$-adapted chart $(U,x_1 ,\dots,x_m )$ and a map 
$\tilde h_s \co U\zfl\GG(n-s)$ like in the proof of Lemma \ref{L41}. Let $C$ be a subset
of $U\zin\zS_s$.

\begin{lemma}\label{L45}
Assume that:

\begin{enumerate} [label={\rm (\arabic{*})}]
\item $C$ is of the first category in $U$.
\item If $p\zpe C$ then any point $q\zpe U$ such that $q_\zlma =p_\zlma$, 
$\zlma=s+1,\dots,m$, belongs to $C$ as well.
\end{enumerate}
Then the set
$$D=\{F\zpe\GG(k+1)\co {\rm there\, exists}\,\, p\zpe C \, {\rm such\, that}\,  
\dim(h_s (p)\zin F)\zmai 1 \}$$
is of the first category in $\GG(k+1)$.
\end{lemma}

\begin{proof}
As $h_s$ does not depend on $(x_1 ,\dots,x_s )$, the set $D$ does not change if one
replaces the statement $p\zpe C$ by $p\zpe T\zin C$, where $T$ is a transversal defined by
$x_1 =c_1 ,\dots,x_s =c_s$ for suitable $c_1 ,\dots,c_s \zpe\RR$.
By Assumption (2), $T\zin C$ is of the first category in $T$.

By Lemma \ref{L21} Part (ii) applied to $\tilde h_s \co T\zfl\GG(n-s)$ and $T\zin C$, it
suffices to check that
$$\dim T\zmei (n+1-(n-s)-(k+1))=(m-k)/2$$
which is just the case since $\dim T=(m-k)/2$.
\end{proof}

By covering $\zS_s -{\zmontar{\circ}\over{\zS}_s}$ by a countable family of $\zS_s$-adapted
charts from Lemma \ref{L45} follows:

\begin{corollary}\label{C46}
The set
$$D_1 =\{F\zpe\GG(k+1)\co {\rm there\, exists}\,\, p\zpe\zS_s -{\zmontar{\circ}\over{\zS}_s} 
 \, {\rm such\, that}\,  \dim(h_s (p)\zin F)\zmai 1 \}$$
is of the first category in $\GG(k+1)$.
\end{corollary}

By Lemma \ref{L22} the set $\W-\U$ is of the first category in $\W$ and $\F$-saturated.
Therefore by covering $\W-\U$ by a countable family of $\zS_s$-adapted charts
whose domains are included in $\W$ from Lemma \ref{L45} follows:

\begin{corollary}\label{C47}
The set
$$D_2 =\{F\zpe\GG(k+1)\co {\rm there\, exists}\,\, p\zpe\W-\U 
 \, {\rm such\, that}\,  \dim(h_s (p)\zin F)\zmai 1 \}$$
is of the first category in $\GG(k+1)$.
\end{corollary}

\begin{lemma}\label{L48}
There exists a set $D_3 \zco\GG(k+1)$ of the first category such that if
$F\zpe \GG(k+1) - D_3$ then $\dim F(p)=k+1$ for every 
$p\zpe\zung_{r>s}\zS_r$.
\end{lemma}

\begin{proof}
Consider a $\zS_r$-adapted chart, $r>s$, $(U,x_1 ,\dots,x_m )$ and let
$\tilde h_r \co U\zfl\GG(n-r)$ be like in the proof of Lemma \ref{L41}.
By Lemma \ref{L21} Part (i) the set
$$E=\{F\zpe\GG(k+1)\co\tilde h_r \co U\zfl\GG(n-r)\, {\rm is\, not\, transverse\, to}\, 
N_F (n-r)\}$$
is of the first category.

As the minimal codimension of the strata of $N_F (n-r)$ is $r-k$, 
$\rm{rank}\, \tilde h_r \zmei m-r$ and $m-r <r-k$ since $r>s$, if $\tilde h_r$ is transverse
to $N_F (n-r)$ then $\tilde h_r (U)\zin N_F (n-r)=\zva$. Therefore if $F\znope E$ then
$\dim F(p)=k+1$ for any $p\zpe U\zin\zS_r$.

Finally, since $\zung_{r>s}\zS_r$ can be covered by a countable family of adapted charts, 
it suffices to set $D_3$ equal to the union of the sets $E$ associated to the elements
of the family.
\end{proof}

Let $D_4 =\{F\zpe\GG(k+1)\co h_s \co{\zmontar{\circ}\over{\zS}_s}     
\zfl\GG(n-s)\, {\rm is\, not\, transverse\, to}\, N_F (n-s)\}$

We claim that any $F\zpe\GG(k+1)-D_1 \zun D_2 \zun D_3 \zun D_4$ is suitable.

Indeed, the first part of Assertion (1) of Proposition \ref{P44} is obvious and for the second one
it is enough to remark that the rank of $h_s \zmei m-s=(m-k)/2$ while, if $\zlma\zmai 2$, 
the codimension of $N_F (n-s,\zlma)\zmai m-k+2$.

Now suppose that $\dim F(p)\zmei k$ for some $p\zpe M$. Then from Corollary \ref{C46} and
Lemma \ref{L48} follows that $p\zpe{\zmontar{\circ}\over{\zS}_s}$; therefore
$p\zpe h^{-1}_s (N_F (n-s))= h^{-1}_s (N_F (n-s,1))$ and $\dim F(p)=k$. Now it is clear that
$h^{-1}_s (N_F (n-s))$ is the singular set of $F$ and, as a consequence, a closed subset 
of $M$.

Since $h_s$ is transverse to $N_F (n-s,1)$ and the codimension of this last submanifold equals
$s-k=m-s$, it follows that $h^{-1}_s (N_F (n-s))= h^{-1}_s (N_F (n-s,1))$ is a regular
submanifold  of $M$ of dimension $s$ (or empty).

For finishing one has to show that if $p\zpe h^{-1}_s (N_F (n-s))$ then $p\zpe\U$.
First observe that if $T$ is a small transversal to $\F$ passing through $p$ then $h_s |T$
is injective (indeed, $(h_s)_* (T_p \F)=0$ and ${\rm rank}\, h_s =m-s$). As $h_s$ is
constant along the leaves of $\F$, if a leaf intersects $T$ more than one time then
 $h_s |T$ is not injective {\em contradiction}. Thus $T$ is neat and $p\zpe\W$. 
Finally,  $p\zpe\U$ because $p\znope\W-\U$ by Corollary \ref{C47}.

\subsection{A construction for the suitable elements of $\GG(k+1)$} \label{sec-42}
Consider a suitable $F\zpe\GG(k+1)$. If $h_s^{-1}(N_F (n-s))=\zva$ then $\dim F(p)=k+1$
everywhere and ${\rm rank} M\zmai k+1$, {\em contradiction}. Therefore assume 
$h_s^{-1}(N_F (n-s))\znoi\zva$ from now on. Let $\{P_\zl \}$, $\zl\zpe L$, be the family
of connected component of  $h_s^{-1}(N_F (n-s))$. Note that $L$ is countable and 
non-empty, and every $P_\zl$ a closed subset and a
regular submanifold of $M$ of dimension $s$.
Moreover each $P_\zl$ is a cylinder of type let us say $r_\zl$.

By Theorem \ref{T23}, for every $\zl\zpe L$ one may identify an open set $A_\zl$ of $M$
including $P_\zl$ with $D^\zl \zpor C_{r_\zl}$, in such a way that
$P_\zl =\{0\}\zpor C_{r_\zl}$ and the fundamental vector fields of the action write
$$X=\zsu_{j=1}^{r_\zl} f_j^{\zl} (x){\frac{\zpar} {\zpar\zh_j}}   
+\zsu_{\zlma=1}^{s-r_\zl} g_\zlma^{\zl} (x){\frac{\zpar} {y_\zlma}}\, .$$

(Here one makes use of coordinates $(x,\zh,y)$ on 
$\RR^{m-s}\zpor\TT^{r_\zl}\zpor\RR^{s-r_{\zl}}$ instead of $(x^\zl ,\zh^\zl ,y^\zl )$
for avoiding an over-elaborate notation.)

\begin{remark}\label{R49}
In the non-compact case, given two different elements $\zl,\zm$ of $L$ it can happen that
$A_\zl \zin A_\zm \znoi\zva$ regardless of the size of the disks $D^\zl$ and $D^\zm$.
For instance, in $\RR^2 -\{0\}$ consider a complete vector field $Z=\zf(x)\zpar/\zpar x_2$
with no zeroes and the action of $\RR$ associated. Then the ``cylinders'' $\{0\}\zpor\RR^+$
and $\{0\}\zpor\RR^-$ cannot be separated by any couple of disks.
\end{remark}

The next step will be to choose a suitable compact set in each $P_\zl$ and separate them.
Consider an injective map $\za\co L\zfl\ZZ$ and the continuous map
$$\zb\co h^{-1}_s (N_F (n-s))=\zung_{\zl\zpe L}P_\zl\zfl\RR$$
that equals $\za(\zl)$ on each $P_\zl$. As $M$ is a normal space $\zb$ prolongs to a
continuous map $\tilde\zb \co M\zfl\RR$. Set 
$\tilde A_\zl =\tilde\zb^{-1}((\za(\zl)-1/4,\za(\zl)+1/4))$. Then $\{\tilde A_\zl \}$,
$\zl\zpe L$, is a  locally finite family of disjoint open sets 
and $P_\zl \zco\tilde A_\zl$, $\zl\zpe L$.

Set $L_0 =\{\zl\zpe L\co r_\zl <s\}$. In each $P_\zl$ consider the compact set 
$K_\zl =\{0\}\zpor\TT^{r_\zl}\zpor\{0\}$ if $\zl\zpe L_0$ and 
$K_\zl =\{0\}\zpor\TT^s$ if $\zl\znope L_0$.
Then $K_\zl \zco \tilde A_\zl$ and, since $K_\zl$ is compact, there are two balls
$B_{a_\zl}(0)\zco\RR^{m-s}$ and $B'_{b_\zl}(0)\zco\RR^{s-r_\zl}$ such that the open set
$A'_\zl =B_{a_\zl}(0)\zpor\TT^{r_\zl}\zpor B'_{b_\zl}(0)$ if $\zl\zpe L_0$, or
$A'_\zl =B_{a_\zl}(0)\zpor\TT^s$ if $\zl\znope L_0$, is includes in $\tilde A_\zl$.
Therefore (thought as subsets of $M$) 
$A'_\zl \zin A'_\zm =\zva$ if $\zl\znoi\zm$.

Finally observe that, up to homotheties centered at the origin of $\RR^{m-s}$ and
$\RR^{s-r_\zl}$ respectively, one may suppose, without loss of generality, that 
$A'_\zl =B_{3}(0)\zpor\TT^{r_\zl}\zpor B'_{3}(0)$ if $\zl\zpe L_0$ and that
$A'_\zl =B_{3}(0)\zpor\TT^s$ if $\zl\znope L_0$.

Clearly, $\{A'_\zl \}$, $\zl\zpe L$, is a locally finite family

\subsection{End of the proof of Theorem \ref{T11}} \label{sec-43}
First one will construct an open subset of $M$ which is diffeomorphic to $M$. 
For every $\zl\zpe L_0$ let $\widetilde X_\zl$ be
a vector field on $\RR^{s-r_\zl}$ such that 
$\widetilde\zF^\zl_1 (\RR^{s-r_\zl})=B'_1 (0)$ where $\widetilde\zF^\zl_t$ is the
flow of $\widetilde X_\zl$ (its existence is assured by Lemma \ref{L33}).
In a natural way $\widetilde X_\zl$ induces a vector field $X_\zl$ on 
$P_\zl =\{0\}\zpor\TT^{r_\zl}\zpor\RR^{s-r_\zl}$ such that 
$\zF^\zl_1 (\{0\}\zpor\TT^{r_\zl}\zpor\RR^{s-r_\zl})=
\{0\}\zpor\TT^{r_\zl}\zpor B'_1 (0)$ where $\zF^\zl_t$ is the flow of $X_\zl$.
 
Let $X$ be the vector field on $\zung_{\zl\zpe L_0}P_\zl$ that equals $X_\zl$ on each
$P_\zl$, $\zl\zpe L_0$. Observe that $\zung_{\zl\zpe L_0}P_\zl$ is a closed subset and a regular
submanifold of $M$; therefore from Corollary \ref{C32} applied to $X$ follows that
$$M_1 =\left(M-\zung_{\zl\zpe L_0}P_\zl\right)\zung\left(\zung_{\zl\zpe L_0}
(\{0\}\zpor\TT^{r_\zl}\zpor B'_1 (0))\right),$$
where each $\{0\}\zpor\TT^{r_\zl}\zpor B'_1 (0)$ has to be regarded as a subset
of $P_\zl$, is diffeomorphic to $M$.

Now $A'_\zl \zin M_1 =B_{3}(0)\zpor\TT^{r_\zl}\zpor B'_{3}(0)-
\{0\}\zpor\TT^{r_\zl}\zpor (B'_{3}(0)-B'_1(0))$ if $\zl\zpe L_0$ and
$A'_\zl \zin M_1=B_{3}(0)\zpor\TT^{s}$ if $\zl\znope L_0$. 
 
Moreover the set of singular points of $F$ belonging to $A'_\zl \zin M_1$ equals
$\{0\}\zpor\TT^{r_\zl}\zpor B'_1 (0)$ when $\zl\zpe L_0$ and 
$\{0\}\zpor\TT^s$ if $\zl\znope L_0$.

By Lemma \ref{L34} there exist $\ze>0$ and a diffeomorphism between 
$$E_\zl =B_3 (0)\zpor B'_3 (0) -\{0\}\zpor (B'_3 (0)-B'_1 (0))$$ 
and $E'_\zl =E_\zl -{\overline B}_1 (0) \zpor S'_1$, $\zl\zpe L_0$, which equals the identity if
$\| x\|\zmai 3-\ze$ or if $\| y\|\zmai 3-\ze$.

Therefore there exists a diffeomorphism between
$$\widetilde E_\zl =B_3 (0)\zpor\TT^{r_\zl}\zpor B'_3 (0) 
-\{0\}\zpor\TT^{r_\zl}\zpor (B'_3 (0)-B'_1 (0))$$ and
 $\widetilde E'_\zl =\widetilde E_\zl -{\overline B}_1 (0) \zpor\TT^{r_\zl}\zpor S'_1$,
which equals the identity if $\| x\|\zmai 3-\ze$ or if $\| y\|\zmai 3-\ze$.

Set 
$$M_2  =\left( M_1 -\zung_{\zl\zpe L_0}\widetilde E_\zl \right)
\zung\left( \zung_{\zl\zpe L_0}\widetilde E'_\zl \right)$$
Then $M_2$ is an open set of $M$ that is diffeomorphic to $M_1$, so to $M$. Indeed,
the diffeomorphism above between $\widetilde E_\zl$ and $\widetilde E'_\zl$ for every
$\zl\zpe L_0$ extends by the identity on 
$$M_1  -\zung_{\zl\zpe L_0}\widetilde E_\zl =M_2  -\zung_{\zl\zpe L_0}\widetilde E'_\zl$$
to a diffeomorphism between $M_1$ and $M_2$.

Observe that now
$$A'_\zl \zin M_2 =B_3 (0)\zpor\TT^{r_\zl}\zpor B'_3 (0) 
-\left[(\{0\}\zpor\TT^{r_\zl}\zpor (B'_3 (0)-B'_1 (0)))   
\zun({\overline B}_1 (0) \zpor\TT^{r_\zl}\zpor S'_1 ) \right] $$
if $\zl\zpe L_0$ and $A'_\zl \zin M_2=B_{3}(0)\zpor\TT^{s}$ if $\zl\znope L_0$. 

Besides, as before, the set of singular points of $F$ belonging to $A'_\zl \zin M_2$ equals
$\{0\}\zpor\TT^{r_\zl}\zpor B'_1 (0)$ when $\zl\zpe L_0$ and 
$\{0\}\zpor\TT^s$ otherwise.

Observe that $\overline B_{1/2}(0)\zpor\TT^{r_\zl}\zpor B'_1 (0)$, $\zl\zpe L_0$, and
$\overline B_{1/2}(0)\zpor\TT^s$, $\zl\znope L_0$, are closed subsets of $M_2$.

The next step will be to modify the (infinitesimal) action of $F$  in order to eliminate its singularities
in $M_2$. Thus for every $\zl\zpe L$ we need to modify the action on $A'_\zl \zin M_2$.
Note that these modifications are compatible among them if their supports are contained in
$B_{1/2}(0)\zpor\TT^{r_\zl}\zpor B'_1 (0)$, $\zl\zpe L_0$, or in
$ B_{1/2}(0)\zpor\TT^s$, $\zl\znope L_0$.

First assume that $\zl\zpe L_0$. Let $Y_1 ,\dots, Y_s$ be the vector fields on $A'_\zl \zin M_2$
defined by $Y_1 =\zpar/\zpar \zh_1$, ..., $Y_{r_\zl} =\zpar/\zpar \zh_{r_\zl}$,
$Y_{r_\zl +1}=\zpar/\zpar y_1$, ..., $Y_s =\zpar/\zpar y_{s-r_\zl}$. Then on $A'_\zl \zin M_2$
the fundamental vector fields of the action write 
$\zsu_{j=1}^s h_j (x)Y_j$.

Let $\{e_1 ,\dots, e_{k+1}\}$ be  a basis of $F$ such that $e_{k+1}$ is a basis of
$F\zin h_s (\{0\}\zpor\TT^{r_\zl}\zpor B'_1 (0))$. (Recall that $h_s$ is constant on
$\{0\}\zpor\TT^{r_\zl}\zpor B'_1 (0)$.) Then there exist vectors
$u_1 ,\dots, u_{s-k} \zpe V$ such that the vector subspace spanned by 
$\{e_1 ,\dots, e_{k}, u_1 ,\dots, u_{s-k}\}$ is supplementary to
$h_s (\{0\}\zpor\TT^{r_\zl}\zpor B'_1 (0))$. Therefore the vector fields $Z_1 =X_{e_1}$,
..., $Z_k =X_{e_k}$, $Z_{k+1}=X_{u_1}$, ..., $Z_s=X_{u_{s-k}}$ are linearly independent
at every point of $\{0\}\zpor\TT^{r_\zl}\zpor B'_1 (0)$.

On the other hand, $Z_\zlma =\zsu _{j=1}^s p_{\zlma j}(x)Y_j$, $\zlma=1,\dots, s$, where
the matrix $(p_{\zlma j}(0))$, $\zlma,j=1,\dots,s$, is invertible. For each $x\zpe B_{1/2}(0)$
set $P(x)=(p_{\zlma j}(x))$, $\zlma,j=1,\dots,s$. 

As $P(0)$ is invertible, there are $\ze>0$ small enough and a family of $s\zpor s$ matrices
$Q(x)=(q_{\zlma j}(x))$, $x\zpe  B_{1/2}(0)$, depending differentiably on $x$ such that:

\begin{enumerate} [label={\rm (\arabic{*})}]
\item Every $Q(x)$, $x\zpe  B_{1/2}(0)$, is invertible.
\item $Q(x)=P(x)^{-1}$ if $\|x\|\zmei\ze$.
\item $Q(x)=Id$ if $\|x\|\zmai 2\ze$.
\end{enumerate}

On $B_{1/2}(0)\zpor\TT^{r_\zl}\zpor B'_1 (0))$ one modifies the infinitesimal action of $V$
by replacing each $$X_v =\zsu_{j=1}^s h_j^v Y_j ,$$ $v\zpe V$, by
$$X_v^Q =\zsu_{j=1}^s \left(\zsu_{\zlma=1}^s h_\zlma^v  q_{\zlma j}\right) Y_j .$$ 

Obviously ``being Abelian'' is not lost because $Q$ only depends on $x$.

Since $Q=Id$ on $B_{1/2}(0)-B_{2\ze}(0)$, the modified action extends to the whole $M_2$.
Moreover the isotropy of the modified action equals that of the original one.

Observe that on $B_\ze (0)\zpor\TT^{r_\zl}\zpor B'_1 (0)$
with the modified action the vector field corresponding to $e_j$, $j=1,\dots,k$,
is $Y_j$. Thus one may suppose, without loss of generality, that on
$B_{\ze}(0)\zpor\TT^{r_\zl}\zpor B'_1 (0)$ the vector fields $X_{e_1},\dots,X_{e_k}$
do not depend on $x$, so they commute with $\zpar/\zpar x_1$.

Consider a function $\zf\co B_{1/2}(0)\zfl\RR$ such that 
$\zf(B_{1/2}(0)-B_{\ze/2}(0))=0$ and $\zf(0)=1$. Then on 
$B_{1/2}(0)\zpor\TT^{r_\zl}\zpor B'_1 (0))$ one modifies the action of $F$ by associating 
to $e_\zlma$, $\zlma=1,\dots,k$ the vector field $X_{e_\zlma}$, and to $e_{k+1}$ the
vector field $X_{e_{k+1}} +\zf\zpar/\zpar x_1$. Clearly:
\begin{itemize}
\item The new action of $F$ is Abelian and has no singularities since 
$$X_{e_1},\dots,X_{e_k},X_{e_{k+1}}+\zf\zpar/\zpar x_1$$
commute and are linearly independent at every point of 
$B_{1/2}(0)\zpor\TT^{r_\zl}\zpor B'_1 (0)$. 
\item It extends to the whole $M_2$.
\end{itemize}
 The case  $\zl\znope L_0$ is very similar and the details are left to readers (it suffices to
 delete the third factor $B'_1 (0)$).

In short, we have constructed an action of $F$ on $M_2$ with no singularities. Thus  
${\rm rank}\, M_2 \zmai k+1$, hence ${\rm rank}\, M\zmai k+1$ {\em contradiction}.
{\em Now the proof of Theorem \ref{T11} is finished.}

\section{Examples} \label{sec-5}
In the next three examples one makes use of the characteristic classes for bounding from
above the minimal dimension of the orbits of a $\RR^n$-action.

\begin{example}\label{E51}
Set $M=S\zpor\RR$ where $S$ is a compact connected  surface of odd characteristic. Then
(the Stiefel-Whitney class) $w_2 (TM)\znoi 0$ and the span and the rank of $S$ equal 1.
Therefore by Theorem \ref{T11} any action of $\RR^n$ on $M$ has an orbit of 
dimension $\zmei 1$.

Obviously this is the best possible result since every complete and non-singular vector field 
gives rise to an action of $\RR$ on $M$ all whose orbits have dimension one.
\end{example}

\begin{example}\label{E52}
Let $R$ be a compact, connected and orientable surface and let $\zp\co M\zfl R$ be a complex
line bundle with odd Chern class. Then as real manifold $\dim M=4$ and $w_2 (TM)\znoi 0$,
hence the span and the rank of $M$ are $\zmei 2$.

Therefore by Theorem \ref{T11} any action of $\RR^n$ on $M$ has some orbit of 
dimension $\zmei 2$.

\begin{lemma}\label{L53}
Consider a real $\zlma$-plane bundle $\zp\co P\zfl Q$ where $Q$ is a manifold. Then 
there exists an action of some $\RR^n$ on $P$ whose orbits are the fibres 
of $\zp\co P\zfl Q$. 
\end{lemma}

\begin{proof}
First observe that a section $g\co Q\zfl P$ gives rise to a vector field $X_g$ on $P$ tangent 
to the fibres and constant along them (with respect to the structure of vector space)
by setting $X_g (g(q))=g(q)$, $q\zpe Q$; moreover $X_g$ is complete.

On the other hand if $h\co Q\zfl P$ is another section then $[X_g ,X_h ]=0$.  

Finally consider a family of sections $\{g_1 ,\dots,g_n \}$ such that 
$\{g_1 (q),\dots,g_n (q)\}$ spans $\zp^{-1}(q)$, as real vector space, 
for every $q\zpe Q$. If $\{e_1 ,\dots,e_n \}$ is a basis of the Lie algebra
$V$ of $\RR^n$, then there exists an action of $\RR^n$ on $P$  such that the fundamental
vector field associated to $e_j$ is $X_{g_j}$,  $j=1,\dots,n$. The orbits of this $\RR^n$
actions are the fibres of $\zp\co P\zfl Q$.
\end{proof}

Now Lemma \ref{L53} applied to $\zp\co M\zfl R$ shows the existence of actions of $\RR^n$
on $M$ whose orbits have dimension two since they are the fibres of this fibre bundle.
\bigskip

Recall that $\CC P^r$ minus one point is (diffeomorphic to) the total space of the canonical
complex line bundle over $\CC P^{r-1}$. On the other hand $\CC^r$ will be regarded
as an open set of $\CC P^r$ in the usual way.

Let $M_\zlma$, $\zlma\zmai 1$, be $\CC P^2$ minus $\zlma$ different points, and
let $J$ be the canonical complex structure on $\CC P^r$. On $\CC P^2$ consider
the projective vector field $X$ which on $\CC^2$ is written as 
$X=\zpar/\zpar z_1 +z_1 \zpar/\zpar z_2$. Then $X$ only has one singular point, let us say 
$p$ (corresponding to the line $z_1 =0$ of $\CC P^2 -\CC^2 \zeq\CC P^1$).

Moreover from the real viewpoint $X,JX$ commute and are linearly independent everywhere
but $p$. Since $M_1$ can be identify to $\CC P^2 -\{p\}$ it follows that 
${\rm rank}\, M_1 \zmai 2$ and finally, as $M_1$ fibres over $\CC P^1$ with odd Chern
class, that ${\rm rank}\, M_1 =2$. 

For constructing an action of $\RR^4$ on $M_1$ it suffices to consider $\{X,JX,Y,JY\}$,
 where $Y$ is the projective vector field that on $\CC^2$ writes
$Y=\zpar/\zpar z_2$, as a basis of the Lie algebra of
fundamental vector fields. This action has one orbit of (real) dimension four ($\CC^2$) 
and one orbit of dimension two ($\CC P^1 -\{p\}$).

A similar construction can be done for  $M_2$ and $M_3$ by considering the projective vector fields
$X',Y',X'',Y''$ respectively, which on $\CC^2$ are written as
$X'=\zpar/\zpar z_1 +z_2 \zpar/\zpar z_2$, $Y'=z_2 \zpar/\zpar z_2$,
$X''=z_1 \zpar/\zpar z_1 -z_2 \zpar/\zpar z_2$ and
$Y''=z_1 \zpar/\zpar z_1 +z_2 \zpar/\zpar z_2$.

\begin{remark}\label{R54}
\begin{itemize} $\,$
\item In fact one has showed that $M_1$, $M_2$ and $M_3$ have file 2.
\item Any holomorphic action of $\RR^n$ on a complex compact manifold of real dimension  4
and non-vanishing characteristic always has fixed points (see Corollary 1.8 of \cite{HT}; for
the analytic actions of $\RR^2$ on dimension 4 see \cite{BO}).
\end{itemize}
\end{remark}

From the next lemma follows that every manifold $M_\zlma$, $\zlma\zmai 1$, has rank 2.

\begin{lemma}\label{L55}
Consider a connected open manifold $P$ of dimension $\zmai 2$ and a point $p\zpe P$.
Then $P$ and $P-\{p\}$ have the same rank and the same span. 
\end{lemma}

\begin{proof}
Consider a non-singular function $f\co P\zfl\RR$. Let $Y$ be the gradient vector field of $f$
with respect to a Riemannian metric such that $Y$ is complete. Let $\zg\co\RR\zfl P$ be any
integral curve of $Y$; then $\zg$ is injective and $\zg(\RR)$ is a closed subset and a regular 
submanifold of $P$.

Set $C=\zg(\RR- (-1,1))$; then $P$ and $P-C$ are diffeomorphic (apply Corollary \ref{C32}
and Lemma \ref{L33}).

Now set $p=\zg(2)$ and let $\zlma$ be either the rank or the span depending on cases.
One has:
$$\zlma(P)=\zlma(P-C)\zmai\zlma(P-\{p\})\zmai\zlma(P).$$
\end{proof}

Observe that the file of $P$ and that of $P-\{p\}$ may be quite different, for instance 
${\rm file}\,\RR^3 =3$ but ${\rm file}(\RR^3 -\{0\})=1$ (see \cite{RO}).
\end{example}

\begin{example}\label{E56}
In Examples \ref{E51} and  \ref{E52} the main tool for bounding the minimal dimension of
the orbits was the Stiefel-Whitney classes. Now we will consider the Pontrjagin classes 
(sometimes all the Stiefel-Whitney classes but $w_0$ vanish, for instance $T\CC P^7$). 

Let $M$ be a connected manifold of dimension $m$ and rank $k$. Consider vector fields
$X_1 ,\dots,X_k$ commuting among them and linearly independent everywhere. Denote by
$\F$ the foliation defined by $X_1 ,\dots,X_k$.
As $T\F$ is parallelizable, $TM$ and the normal bundle to $T\F$ have the same characteristic
classes.

Let $\za$ be a non-zero element, of degree $4\zlma\zmai 4$, of the ring of Pontrjagin classes of
$TM$. By the Bott's theorem on characteristic classes of foliations $4\zlma\zmei 2m-2k$, 
hence $k\zmei m-2\zlma$.

Therefore by Theorem \ref{T11} on has:

\begin{proposition}\label{P57}
Under the hypotheses above any action of $\RR^n$ on $M$ possesses  an orbit of
dimension $\zmei m-\zlma-1$.
\end{proposition}

A particular case of Proposition \ref{P57} is as follows. Let $M$ be the total space of a
complex line bundle over  $\TT^{4r}$, $r\zmai 1$, with Chern class $c_1$ such that
$c_1^{2r}\znoi 0$. Then as real 2-plane bundle $p_1^r \znoi 0$ where $p_1$ is its
first Pontrjagin class. Since the tangent bundle of $\TT^{4r}$ is parallelizable, it is
easily seen that $p_1 (TM)^r \znoi 0$. Now $m=4r+2$ and $\zlma =r$; therefore by
Proposition \ref{P57} any action of $\RR^n$ on $M$ has an orbit of
dimension $\zmei 3r+1$.

Observe that the span of $M$ equals $4r$ (while the rank $\zmei 2r+2$). Indeed, by means of 
a connection the vector fields on $\TT^{4r}$: $\zpar/\zpar\zh_1 ,\dots,\zpar/\zpar\zh_{4r}$ can
be lifted to $M$, so ${\rm span}\,M\zmai 4r$. As $M$ is orientable if  
${\rm span}\,M\zmai 4r+1$ then ${\rm span}\,M=4r+2$ and $M$ is parallelizable, which
implies $p_1 (TM)=0$, {\em contradiction}.

Moreover if $c_1$ is even, i.e. if there is $\zb\zpe H^2 (\TT^{4r},\ZZ)$ such that $c_1 =2\zb$, 
then $w_j (TM)=0$, $j=1,\dots,4r+2$.

Another case is the following one. Let $M'$ be $\CC P^7$ minus a point. Then 
$w_j (TM')=0$, $j\zmai 1$, but $p_1 (TM')^3 \znoi 0$; therefore ${\rm rank}(M')\zmei 8$
and any action of $\RR^n$ on $M'$ possesses an orbit of
dimension $\zmei 10$. By Lemma \ref{L55} and Theorem \ref{T11} the same result holds 
for $\CC P^7$ minus a finite number ($\zmai 1$) of points.
\end{example}

\section{$\RR^n$-actions whose orbits have codimension $\zmei 1$} \label{sec-6}
For the sake of completeness, one will prove the following result:

\begin{proposition}\label{P61}
Consider an action of $\RR^n$ on a connected manifold $M$ of dimension $m\zmai 3$.
Assume that the codimension of every orbit is $\zmei 1$. Then the universal covering
of $M$ is diffeomorphic to $\RR^m$.
\end{proposition}

\begin{proof}
It is enough to prove the result when $M$ is simply connected. 

First one will construct a codimension one foliation on $M$ with no vanishing cycles. 
If $n\zmei m$ set $C=\zva$; otherwise let $C$ be the set of those $F\zpe\GG(m-1)$ such that 
$h_m \co\zS_m \zfl\GG(n-m)$ is not transverse to $N_F (n-m)$ [recall that $\zS_m$ is an open
set of $M$]. By (i) of Lemma \ref{L21} $C$ is of the first category in $\GG(m-1)$.

Observe that if $F$ belongs to $\GG(m-1)-C$ then $\dim F(p)=m-1$ for every $p\zpe\zS_m$.
Indeed, if $n\zmei m$ it is obvious and if $n>m$, since ${\rm rank}\, h_m =0$ and the 
codimension of each stratum of $N_F(n-m)$ is $\zmai 2$, $h_m$ is transverse to $N_F(n-m)$
if and only if $h_m (\zS_m )\zin N_F (n-m)=\zva$.

Now consider a $\zS_{m-1}$-adapted chart $(U,x_1 ,\dots,x_m )$ and an extension
$\tilde h_{m-1}\co U\zfl\GG(n-m+1)$ of $h_{m-1}$ like in the proof of Lemma \ref{L41}.  
Given $a\zpe U$ denoted 
by $S_a$ and $T_a$ the slice and the transversal in $U$ defined by $x_m =a_m$ and by
$x_1 =a_1,\dots,x_{m-1}=a_{m-1}$ respectively. Observe that the set 
$U\zin\zpar\zS_{m-1}$, where 
$\zpar\zS_{m-1}=\zS_{m-1}-{\zmontar{\circ}\over{\zS}_{m-1}}$, is saturated for the slices
of $U$ and of the first category in $U$. Therefore given a transversal $T_a$ the set 
$T_a \zin\zpar\zS_{m-1}$ is of the first category in $T_a$.

Since $\dim T_a =1$, by (ii) of Lemma \ref{L21} applied to 
$\tilde h_{m-1}\co T_a \zfl\GG(n-m+1)$ the set $D$ of those $F\zpe\GG(m-1)$ for which
there exists $p\zpe T_a \zin\zpar\zS_{m-1}$ such that $\dim(h_{m-1}(p)\zin F)\zmai 1$ is of
the first category in $\GG(m-1)$. 

Note that if $F\znope D$ then $\dim F(p)=m-1$ for every
$p\zpe T_a\zin\zpar\zS_{m-1}$. 
But $h_{m-1}$ and $\tilde h_{m-1}$ are constant along the slices of $U$, so  if $F\znope D$ 
then $\dim F(q)=m-1$ for every $q\zpe U\zin\zpar\zS_{m-1}$. 

As $\zS_{m-1}$ can be covered by a countable family of $\zS_{m-1}$-adapted charts, it
follows the existence of a set $D'\zco\GG(m-1)$ of the first category  such that 
$\dim F(p)=m-1$ for any $F\zpe(\GG(m-1)-D')$ and any $p\zpe\zpar\zS_{m-1}$. 

Consider $F\zpe(\GG(m-1)-C\zun D')$; then the singular set $S(F)$ is included in
${\zmontar{\circ}\over{\zS}_{m-1}}$. Let $\F$ be the codimension one foliation on $M$
defined:
\begin{itemize} 
\item  On $M-S(F)$ by $F$, that is to say $T_p \F =F(p)$ for every $p\zpe M-S(F)$.
\item On ${\zmontar{\circ}\over{\zS}_{m-1}}$ by the orbits of the action of $\RR^n$
\end{itemize}

It is easily checked that the definition of $\F$ is coherent.

As $F$ is a subalgebra of the Lie algebra $V$ of $\RR^n$, the foliation $\F$ is also defined
on $M-S(F)$ by an action locally free of $\RR^{m-1}$. Therefore $\F$ does not have
vanishing cycles on $M-S(F)$ (see Chapter 3 of \cite{GO}).

Now consider any point $q\zpe S(F)$. Then there exists $F'\zpe\GG(m-1)$ such that
$F'(q)=T_q \F$. Thus on ${\zmontar{\circ}\over{\zS}_{m-1}}-S(F')$, where $S(F')$
is the singular set of $F'$, the foliation $\F$ is defined by a locally free of
$\RR^{m-1}$ and does not have vanishing cycles.

In short, on $M$ the foliation $\F$ does not possesses vanishing cycles. Thus if $L$ is any 
leaf of $\F$ its fundamental group injects in $\zp_1 (M)=0$, hence $L$ is simply connected.
But by the construction of $\F$ its leaves are cylinders, so $L=\RR^{m-1}$. In other words,
$\F$ is a plane foliation of $M$ and by Corollary 3, page 110, of \cite{PL} $M$ is 
diffeomorphic to $\RR^m$.
\end{proof}

\begin{example}\label{E62}
Let $\RR_e^4$ be any exotic $\RR^4$. Then every action of $\RR^n$ on $\RR_e^4$ 
has an orbit of dimension $\zmei 2$. 
\end{example}







\end{document}